\documentclass{amsart}

\usepackage{amssymb}
\usepackage{bm}
\usepackage{graphicx}
\usepackage[centertags]{amsmath}
\usepackage{amsfonts}
\usepackage{amsthm}

\newtheorem{theorem}{Theorem}
\newtheorem*{mtheo}{Theorem}

\newtheorem{claim}{Claim}
\newtheorem{lemma}{Lemma}

\newtheorem{definition}{Definition}
\theoremstyle{definition}

\begin{document}
 
\title[ An infinite self dual Ramsey theorem]{An infinite self dual Ramsey theorem}

\author{ Dimitris VLITAS}

\address{Department of Mathematics, University of Toronto, 40 St. George Street, Toronto, Ontario, Canada M5S 2E4}
\email{dimitrios.vlitas@utoronto.ca}

\thanks{The research leading to these results has received funding from the [European Community's] Seventh Framework Programme [FP7/2007-2013] under grant agreement n 238381}

\begin{abstract} In a recent paper \cite{So} S. Solecki proved a finite self-dual Ramsey theorem that extends simultaneously the classical finite Ramsey theorem \cite{Ra} and the Graham--Rothschild theorem \cite{Gr--Ro}. In this paper we prove the corresponding infinite dimensional version of the self-dual theorem. As a consequence, we extend the classical infinite Ramsey theorem \cite{Ra} and the Carlson--Simpson theorem \cite{Ca--Si}.
\end{abstract}
\maketitle

\section{Introduction}

 Recall that the classical infinite version of Ramsey's theorem \cite{Ra} states that given any finite coloring of the set of all $K$ element subsets of $\omega$ there exists an infinite subset $A$ of $\omega$ such that the restriction of the coloring on all subsets of $A$ of cardinality $K$ is constant.
 
 The dual form of Ramsey's theorem, the Carlson--Simpson theorem \cite{Ca--Si}, states that given any finite Borel coloring of the set of all $K$-partitions of $\omega$ into $K$ many classes, there exists a partition $r$ of $\omega$ into $\omega$ many classes such that the set of all $K$ partitions of $\omega$ which are coarser than $r$ is monochromatic.
 
 As it is well known the above results have finite versions as well, namely the finite Ramsey theorem and the Graham--Rothschild theorem \cite{Gr--Ro}. Recently S. Solecki proved \cite{So} a self-dual theorem which extends simultaneously the finite version of the Ramsey theorem and the Graham--Rothschild theorem. 
 
 The main goal of the present paper is to obtain an infinite self-dual Ramsey theorem which also extends simultaneously the infinite Ramsey theorem \cite{Ra} and the Carlson-Simpson theorem \cite{Ca--Si}.
 
 We explicit state our contribution in Section $2$. The rest of this introduction is devoted to the presentation of some background material. We also include a brief review of related work.

 \subsection{Basic definitions}
 
      We follow the terminology introduced in \cite{So}. Let $K\leq L\leq \omega$. We view a natural number $K$ as a linear order $(K, \leq_K)$ where $\leq_K$ is the usual ordering on the set $\{ 0, \dots, K-1 \}$. Similarly in the case $K=\omega$.
 By a rigid surjection $t:L\to K$ we mean a surjection with the additional property that images of initial segments of $L$ are also initial segments of $K$. We denote the image of a rigid surjection $t$, by $im(t)$ and its domain by $dom(t)$. Now let $t:L\to K$ and $i:K\to L$ be two maps. We say that the pair $(t,i)$ is a connection if for all $x\in K, y\in L$:\\  \begin{center} $t(i(x))=x$ and if $y\leq i(x)$ then $t(y)\leq x$.\end{center} 
 Note that if $(t,i)$ is a connection, then $t$ is a rigid surjection and $i$ is an increasing injection.

   With this terminology the classical Ramsey Theorem can be stated as follows. Let $l,K$ be natural numbers. For any $l$-coloring of all increasing injections $j:K\to \omega$ there exists an increasing injection $j_0:\omega \to \omega$ such that the set $\{\, j_0\circ j: j:K\to \omega \, \}$ is monochromatic. \\
  
  Similarly, the Carlson--Simpson Theorem can be stated as follows. Let $l$ a natural number. For any Borel $l$-coloring of all rigid surjections $s:\omega \to K$, there exists a rigid surjection $s_0:\omega \to \omega$ such that the set $\{\, s\circ s_0 : s:\omega \to K \, \}$ is monochromatic.\\

 \subsection{ The space of connections}

    Given a finite, possibly empty, alphabet $A=\{\, \alpha_0,\dots, \alpha_{|A|-1}\,\}$, written in an increasing order, and given $K\le L\le \omega$ we define\\ 
    
    $F^A_{L,K}=\{ \, (r,c)$: $ r:A\cup L  \to A \cup K$, $ c: K\to L, r\upharpoonright A=id_A$ and $c$ is an increasing injection such that $ r(c(x))=x, y\leq c(x)\Rightarrow r(y)\leq x$, for all $x\in K, y\in L\, \}$.\\  
   
   Here we view $A\cup L$, and $A\cup K$ respectively, as initial segments of $\omega$, where every element of $A$ preseats every element of $L$, and every element of $A$ preseats every element of $K$ respectively. Notice that $A$ is not in the domain of $c$. We note that if $A$ is empty, we will suppress the superscript $A$ in $F_{K,L}^A$.

   Let $(r,c)\in F_{L,K}^A$, where $K\leq L\leq \omega$. The element $r$ is essentially an equivalence relation on $ A\cup L$ into $A \cup K$ many equivalent classes, where each equivalent class is permitted to contain at most one element of the finite alphabet $A$. On the other hand the map $c$ is a choice function for the set of all $\bf{free}$ classes of $r$ (here by a free class we mean an $r$-equivalent class not containing an element of $A$). The choice function $c$ selects $c(k)$ such that $$\min r^{-1}(\{k\})\leq c(k) < \min r^{-1}(\{k+1\}).$$ 
  For every $n\in K$, let $X_n=r^{-1}(\{n\})$ and set $E_n=\min X_n$. We define $$(r\upharpoonright E_{n},c\upharpoonright n)=(r,c)[n] \in F_{E_{n},n}^A$$ 
 
 Note that $r\upharpoonright E_n$ includes $A$ in its domain (this follows from our convention that on $A\cup L$ every element of $A$ preseats every element of $L$).
 
   \begin{definition} 
   
   Let $K \leq L<\omega$ and $ M\leq N\leq \omega$. We say that $(t,i)\in F^A_{L,K}$ is an initial segment of $(r,c)\in F^A_{N,M}$, if $(t,i)=(r,c)[K]$. We denote the fact that $(t,i)$ is an initial segment of $(r,c)$ by $(t,i)\preceq (r,c)$.
  
   \end{definition}
   
   For $(t,i), (r,c)$ as above, we extend the above definition to the case where $(t,i)=(r\upharpoonright L, i\upharpoonright K)$. We denote this case by $$(t,i)\sqsubseteq (r,c).$$
   
   Observe that $(t,i)\sqsubseteq (r,c)$ does not necessarily implies that $(t,i)\preceq (r,c)$.

   \subsection{Composition, length and reduct}
   
   We are about to introduce three basic operations on the space of connections. First given $(s,j) \in F^A_{L,K}$ and $(r,c) \in F^A_{M,L}$, where $K\leq L\leq M\leq \omega$, we define the $\mathit{composition}$ by $$(s,j)\cdot(r,c)=(s\circ r, c\circ j)\in F_{M, K}^A.$$ Notice that the order of composition in the two coordinates is not the same. 
   
Note also that the multiplication, the way defined above, is nothing more than a way to put equivalence classes together and accordingly adjusting the choice functions.\\
      We proceed with the following definition.
  \begin{definition} Let $K\leq L\leq \omega$ and $(t,i)\in F_{L,K}^A$. We define its length of $(t,i)$ to be the domain of $t$. We shall denote it by $|(t,i)|$.
     \end{definition}
 We remark that in Section $3$ where we will deal with rigid surjections, we will still use the above terminology. Specifically we will talk of the length of a rigid surjection $t$ and we will also denote it by $|t|$.   
  
  We need to introduce one more definition.
   
   \begin{definition}
     Let $(r_0,c_0)\in F_{M,K}^A$, $(r_1,c_1)\in F_{M,L}^A$ with $K\leq L \leq M\leq \omega$. We say that  $(r_0,c_0)$ is a $\mathit{reduct}$ of $(r_1,c_1)$, denoted by $(r_0,c_0)\leq (r_1,c_1)$, if both have the same length and there exists $(r,c)\in F_{L,K}^A$ such that $$(r_0,c_0)=(r,c)\cdot (r_1,c_1).$$
  
         \end{definition}
 We close this subsection by stating the finite self-dual theorem due to Solecki. \cite{So}.
 
 \begin{theorem}
For any $K,l$ and $M$ positive integers, there exists $N\in \omega$ such that for any finite coloring $c:F_{N,K}\to l$ there exists a connection $(s_0,j_0)\in F_{N,M}$ such that the set $$F_{M,K}\cdot (s_0,j_0)=\{\, (s,j)\cdot (s_0,j_0): (s,j)\in F_{N,M}\,\}$$ is monochromatic.
  \end{theorem}

          \subsection{Topology of the space of connections}
         
Let $K<L< \omega$, the space $F_{\omega,\omega}^A$, becomes a topological space, with basic open sets: $$ [ (t,i), (r,c)]^A_{\omega}=\{ (r',c')\in F_{\omega,\omega}^A: (t,i)\sqsubseteq (r',c')\text{ and }(r',c')\leq (r,c) \, \},$$

    where $(t,i)\in F^A_{L,K}$ and $(r,c)\in F_{\omega,\omega}^A$. If $(t,i)=(id_A, \emptyset)$, then we will write $$[(r,c)]^A_{\omega}$$
    Observe that $ [ (t,i), (r,c)]^A_{\omega}=\emptyset$ if and only if there is not $(r',c')\leq (r,c)$ so that $(t,i)\sqsubseteq (r',c')$.
    In the case of $A=\emptyset$, we will omit the superscript $A$.
    On what follows for notational simplicity we will omit the subscript $\omega$ from our definition of basic open sets.

  \begin{center}
   \section{Main theorem}
   \end{center}
   
   The main theorem of this paper is the following.
   
   \begin{theorem}
Given a finite Suslin measurable coloring $g:F_{\omega,\omega} \to l$, there exists a $(r,c)\in F_{\omega,\omega}$ such that $[(r,c)]$ is monochromatic.
\end{theorem}

Recall that a map $f:X \to Y$ between two topological spaces is Suslin measurable, if the preimage $f^{-1}(U)$ of every open subset $U$ of $Y$ belong to the minimal $\sigma-$field of subsets of $X$ that contains its closed sets and it is closed under the Suslin operation \cite{Ke}. At this point we introduce a new topology on the space $F_{\omega,\omega}^A$. Given $(t,i) \in F_{L,K}^A$, $K\leq L<\omega$, let $$[(t,i)]=\{ (r,c)\in F^A_{\omega,\omega}: (t,i)\sqsubseteq (r,c)[K] \}.$$
Observe that this topology is weaker than the one introduced in the subsection $1.4$ above.
By $(F_{\omega,\omega}^A,\tau_2)$, we denote the topological space of $F_{\omega,\omega}^A$ with the weaker topology introduced just above. Next we identify an element $(r,c)$ of $F_{\omega, \omega}^A$ with the pair $((X_{n'})_{n' \in A\cup \omega}, (x_n)_{n\in \omega})$, where $X_{n'}=r^{-1}(\{n'\})$, $n'\in A\cup  \omega$, and $c(n)=x_n$, for $n\in \omega$. Observe that $\cup_{n'\in A\cup \omega} X_n'=A \cup \omega$ and $\min X_n< \min X_{n+1}$ for all $n\in \omega$. Let $\mathcal{C}_A$ be the space of all such pairs. We say that $((X_{n'})_{n'\in A\cup  \omega}, (x_n)_{n\in \omega})\leq ((Y_{n'})_{n'\in A\cup  \omega}, (y_n)_{n\in \omega})$, if and only if $(X_{n'})_{n'\in A\cup \omega}$ is a coarser partition than the $(Y_{n'})_{n'\in A\cup  \omega}$ and $(x_n)_{n\in \omega}$ is the subsequence of $(y_n)_{n\in \omega}$. We turn $\mathcal{C}_A$ to a topological space with basic open sets of the form:\\ 
$[((X_{n'}\upharpoonright E_k)_{n'\in A\cup k}, (x_n)_{ n \in k})]= \{(Y_{n'})_{n'\in A\cup \omega},(y_n)_{n\in \omega}, \text{ so that }((Y_{n'}\upharpoonright E_k)_{n'\in A\cup k}, \\(y_n)_{n\in k})=((X_{n'}\upharpoonright E_k)_{n'\in A\cup k}, (x_n)_{n\in k})\},$\\
where $E_k=\min X_k, k \in \omega$ as introduced in the first section. If $(r,c)$ gives rise to $((X_{n'})_{n' \in A\cup \omega}, (x_n)_{n\in \omega})$ and $(t,i)=((X_{n'}\upharpoonright E_k)_{n'\in A\cup k}, (x_n)_{n\in k})$, $c(n)=x_n, n<k$,
 then $[(t,i), (r,c)]=[((X_{n'}\upharpoonright E_k)_{n'\in A\cup k}, (x_n)_{n\in k}),((X_{n'})_{n'\in A\cup  \omega},(x_n)_{n\in \omega})]$ and $[(t,i)]=[((X_{n'}\upharpoonright E_k)_{n'\in A\cup k}, (x_n)_{n\in k})]$.

 Notice that if $(r,c)\neq (r',c')$ and $((X_{n'})_{n'\in A\cup \omega}, (x_n)_{n\in \omega})$, $((X'_{n'})_{n'\in A\cup \omega}, (x'_n)_{n\in \omega})$ the corresponding pairs under the above identification, then $((X_{n'})_{n'\in A\cup \omega}, (x_n)_{n\in \omega})\\\neq ((X'_{n'})_{n'\in A\cup \omega}, (x'_n)_{n\in \omega})$. As a consequence the mapping $T:F_{\omega, \omega}^A \to \mathcal{C}_A$, that corresponds to the above identification, is one-to-one and onto.
 
Let $\mathcal{P}_A$ be the space of all rigid surjections $ r:A\cup \omega\to A\cup \omega$, where $r\upharpoonright A=id_A$, identified with elements of the form $(X_{n'})_{n'\in A\cup \omega}$, where $X_{n'}=r^{-1}\{ n'\}$, $n'\in A\cup \omega$, $\cup_{n'\in A\cup \omega} X_{n'}=A \cup \omega$ and $\min X_n<\min X_{n+1}$. Once more, $(X_{n'})_{n'\in A\cup \omega}\leq (Y_{n'})_{n'\in A\cup \omega}$ if and only if $(X_{n'})_{n'\in A\cup \omega}$ is a coarser partition than $(Y_{n'})_{n'\in A\cup \omega}$. We turn $\mathcal{P}_A$ into a topological space by introducing a topology with basic open set is of the form: \\$[(X_{n'}\upharpoonright E_k)_{n' \in A\cup k})]=\{(Y_{n'})_{n'\in A\cup \omega}, \text{ so that } (Y_{n'}\upharpoonright E_k)_{n'\in A \cup k}=(X_{n'}\upharpoonright E_k)_{n'\in A\cup k}\},$\\

where $E_k=\min X_k, k \in \omega$. First we will prove the following theorem.

\begin{theorem}
Consider the topological space $(F_{\omega, \omega}^A,\tau_2)$.
Given a finite Borel measurable coloring $g:F_{\omega,\omega}^A \to l$, there exists a $(r,c)\in F_{\omega,\omega}^A$ such that $[(r,c)]$ is monochromatic.
\end{theorem}

Observe that the above theorem, is a special case of Theorem $2$, with $A=\emptyset$, for the weaker topology $\tau_2$.
We are going to derive Theorem $3$, from the Carlson--Simpson theorem. The Carlson--Simpson theorem in the above context is the following (Corollary 4.12 in \cite{Ca--Si}).

 \begin{mtheo}(Carlson--Simpson)
Given a finite Borel measurable coloring $g':\mathcal{P}_A \to l$, there exists a $(X_{n'})_{n'\in A\cup  \omega}\in \mathcal{P}_A$ such that $$\{(Y_{n'})_{n'\in A\cup \omega}: (Y_{n'})_{n'\in A\cup  \omega} \leq (X_{n'})_{n'\in A\cup  \omega}\}$$ is monochromatic.
\end{mtheo}

Let $Q: \mathcal{P}_A \to \mathcal{C}_A$ be the mapping defined by $$Q((X_{n'})_{n'\in A\cup \omega})= ((X_{n'})_{n'\in A} \cup (X_{2n'}\cup X_{2n'+1})_{n'\in \omega}, \\(\min X_{2n+1})_{n\in \omega}).$$ Observe that $Q$ is onto. We claim that: \\$Q([(X_{n'}\upharpoonright E_{2k+2})_{n'\in A\cup 2k+2}]) \supset 
[((X_{n'})_{n'\in A}\cup (X_{2n'}\cup X_{2n'+1}\upharpoonright E_{2k+2})_{n'\in k+1}, \\(\min X_{2n'+1})_{n'\in k+1}))]$.\\ Pick $((Y_{n'})_{n'\in A\cup \omega}, (y_n)_{n\in \omega})\in [((X_{n'})_{n'\in A}\cup (X_{2n'}\cup X_{2n'+1}\upharpoonright E_{2k+2})_{n'\in k+1}, \\(\min X_{2n'+1})_{n'\in k+1}))]$. Let $(X'_{n'})_{n'\in A\cup \omega}$ be defined as follows: $X'_{n'}=X_{n'}$ for all $n' \leq A\cup 2k+1$, $X'_{2k+2}=\cup_{n'\in [m,m')} X_{n'}$, where $m$ is the least such a number so that $y_{k}<\min X_{m} $ and $m'$ is so that $\min X_{m'}=y_{k+1}$. Then $X'_{2k+3}=X_{m'}$, where $\min X_{m'}=y_{k+1}$. Next set $X'_{2k+4}=\cup_{n\in [m'+1,m'')}X_n$, where $\min X_{m''}=y_{k+2}$ and $X'_{2k+5}=X_{m''}$.
 Proceed in this manner. Then $Q((X'_{n'})_{n'\in A\cup \omega})=((Y_{n'})_{n'\in A\cup \omega},(y_n)_{n\in \omega})$.
 
 Next we claim that $Q$ is a Borel measurable mapping. To see this let $E_{k}=\min X_{k}$, $k\in \omega$, and consider the basic open set $[((X_{n'}\upharpoonright E_{k})_{n'\in A\cup k}, (x_n)_{n\in k})]$ in $\mathcal{C}_A$. Then $Q^{-1}([((X_{n'}\upharpoonright E_{k})_{n'\in A\cup k}, (x_n)_{n\in k})])$ is Borel measurable. To see this notice that it can be written as the finite union of basic open sets of $\mathcal{P}_A$, of the form $[(Y_{m'}\upharpoonright E_k)_{m'\in A\cup 2k}]$. Where a possible such a $((Y_{m'}\upharpoonright E_k)_{m'\in A\cup 2k})$ is defined as follows: $Y_0\upharpoonright x_0=X_0\upharpoonright x_0$, $Y_1=(X_0\upharpoonright  E_1)\setminus (X_0\upharpoonright x_0)$. In general $Y_{m'}=(X_{m'-1}\upharpoonright E_{m'})\setminus (X_{m'-1}\upharpoonright x_{m'-1})$, for $0<m'<2k$, odd and $Y_{m'}=X_{m'-1}\upharpoonright x_{m'-1}$ for $0<m'<2K$ even. Also $Y_{m'}=X_{m'}$ for $m'\in A$.
 Any $(Y_{m'}\upharpoonright E_k)_{m'\in A\cup 2k}$ so that $((Y_{m'}\upharpoonright E_k)_{m'\in A}\cup (Y_{2m'}\cup Y_{2m'+1}\upharpoonright E_k)_{0\leq m' < 2k }, (\min Y_{2m'+1})_{m'\in 2k})=((X_{n'}\upharpoonright E_{k})_{n'\in A\cup k}, (x_n)_{n\in k})$ will do. Observe that there are only finitely many possible candidates for $(Y_{m'}\upharpoonright E_k)_{m'\in A\cup 2k}$. Therefore the inverse image of a basic open set of $\mathcal{C}_A$ is the finite union of basic open sets of $\mathcal{P}_A$.

Now we proceed to the proof of Theorem $3$.

\begin{proof} Let $F_{\omega, \omega}^A= C_0\cup \dots \cup C_{l-1}$ be a Borel measurable coloring with respect the topology $(F_{\omega, \omega}^A,\tau_2)$. Consider the corresponding coloring of $\mathcal{P}_A$ under $Q$, $\mathcal{P}_A=C'_0\cup \dots \cup C'_{l-1}$, where $(X_{n'})_{n'\in A\cup \omega}\in C'_i$ if and only if $Q((X_{n'})_{n'\in A\cup \omega})\in C_i$, for $i<l$. By the Carlson--Simpson theorem we get an element $(X^0_{n'})_{n'\in A\cup \omega}$ so that the set $A'=\{(Y_{n'})_{n'\in A\cup  \omega}\leq (X^0_{n'})_{n'\in A\cup \omega}\}$ is monochromatic. In other words there exists $k<l$ so that every element of $A'$ goes to $C'_k$. As a result the set $A=\{((X'_{n'})_{n'\in A})\cup (X'_{2n'}\cup X'_{2n'+1})_{n'\in \omega},(\min X'_{2n'+1})_{n'\in \omega})): (X'_{n'})_{n'\in A\cup \omega}\in A'\}$ is monochromatic as well. Namely every member of $A$ goes to $C_k$ for the same fixed $k<l$. Therefore $Q((X^0_{n'})_{n'\in A\cup \omega})$ satisfies the conclusions of our theorem. To see this notice that for every $((Y_{n'})_{n'\in A})\cup(Y_{2n'}\cup Y_{2n'+1})_{n'\in \omega},(\min Y_{2n'+1})_{n\in \omega}))\leq Q((X^0_{n'})_{n'\in A\cup \omega})$ is an element of $A$.

\end{proof}

 Next show that $F_{\omega, \omega}$ forms a topological Ramsey space. The reader at this point is assumed to be familiar with the theory of Ramsey spaces as introduced and developed by S. Todorcevic in \cite{To}.
   We have already introduced a topology on $F_{\omega,\omega}$, the one with basic open sets $[(t,i),(r,c)]$, where $(t,i)\in F_{L,K}$ and $(r,c)\in F_{\omega, \omega}$.  We say that a subset $\mathcal{X}$ of $F_{\omega,\omega}$ is \emph{Ramsey} if for every $[(t,i),(r,c)]\neq \emptyset$ there is a $(r',c')\in [(t,i),(r,c)]$ such that  either $[(t,i),(r',c')]\subset \mathcal{X}$ or $[(t,i),(r',c')]\subset \mathcal{X}^c$, and    $\mathcal{X}$  is \emph{Ramsey null} if for every $[(t,i),(r,c)]\neq \emptyset$, there is $(r',c')$ such that $[(t,i),(r',c')]\cap \mathcal{X}=\emptyset$. We are going to see that Ramsey subsets of $F_{\omega,\omega}$ are exactly those with the Baire property. Moreover we will show that meager sets are Ramsey null. In order to see the equivalence of those topological and Ramsey notions we use Theorem 5.4 from \cite{To}. To apply the general theory we need to define finite approximation of connections and we need to verify axioms $A.1$ to $A.4$ from \cite{To}.\\
     
    We define finite approximations  $$ u:F_{\omega,\omega} \times \omega \to \mathcal{A}F_{\omega,\omega}$$ to connections as follows: $$u((r,c),n)=u_n((r,c))=(r,c)[n].$$ In other words $u_n((r,c))$ is the initial segment $(r,c)[n]$ introduced in the first section. The image of $u_n$ is denoted by $({\mathcal{A}F}_{\omega, \omega})_n$ and the union $\cup_{n\in \omega} ({\mathcal{A}F}_{\omega, \omega})_n= {\mathcal{A}F}_{\omega, \omega}$ forms the set of all finite approximations of all elements of the space $F_{\omega, \omega}$.
 
  We pass now to state and verify the axioms A.1--A.4. The first three are immediate consequence of the definitions:\\
 
    $\boldsymbol{A.1.}$
    
     Let $(r,c), (r',c')\in F_{\omega, \omega}$.
     
    \begin{enumerate}
        \item{} $u_0((r,c))=\emptyset$ for all $(r,c)\in F_{\omega,\omega}$.\\
        \item{} $(r',c')\neq (r,c)$ implies $u_n((r,c))\neq u_n((r',c'))$ for some $n\in \omega$.\\
       \item{}  $u_n((r,c))=u_m((r',c'))$ implies $n=m$ and $u_k((r,c))=u_k((r',c'))$ for all $k<n$.
   \end{enumerate}

  $\boldsymbol{A.2.}$
  
   Given $(t',i')\in F_{L,K},(t,i)\in F_{N,M}$ we define $(t',i')\leq_{fin}(t,i)$ if they have the same length, i.e., $N=|(t,i)|=|(t',i')|=L$ and $(t',i')\leq (t,i)$. In other words $(t,',i')\leq_{fin}(t,i)$ iff $(t',i')$ is a reduct of $(t,i)$ as defined in section $1$, Definition $3$. 
   
  \begin{enumerate}
  
    \item{} For any $(t,i)\in F_{N,M}$ the set\\ $\{\, (t',i')\in F_{N,K}:K\leq M, (t',i')\leq_{fin}(t,i)\, \}$ is finite.\\
  \item{} For any $(r',c'), (r,c)\in F_{\omega, \omega}$, $(r',c')\leq (r,c)$ if and only if \\ $(\forall n)(\exists m) u_n((r',c'))\leq_{fin}u_m((r,c))$.\\
  \item{}  For all $(t,i),(t',i'):$\\
  $[(t',i')\preceq (t,i)\wedge (t,i)\leq_{fin}(t'',i'') \to \exists (\tilde{t},\tilde{i})\leq_{fin} (t'',i''):\, (t',i')\leq_{fin}(\tilde{t},\tilde{i})]$.\\
  \end{enumerate}
  
  $\boldsymbol{A.3}$
  
  Let $(t,i)\in F_{M,N}$, $(r,c),(r',c'), (r'',c'')\in F_{\omega, \omega}$.
  \begin{enumerate}
  \item{} If $[(t,i),(r,c)]\neq \emptyset $ then $[(t,i),(r',c')]\neq \emptyset$ for all $(r',c')\in [(t,i),(r,c)]$.\\
  \item{} $(r',c')\leq (r,c)$ and $[(t,i),(r',c')]\neq \emptyset$ imply that there is $(r'',c'')\in [(t,i),(r,c)]$ such that $\emptyset \neq [(t,i),(r'',c'')]\subseteq [(t,i),(r',c')]$.\\
  \end{enumerate}
  The last axiom A.4 is less obvious.
  
 $ \boldsymbol{A.4}$\\
 Let $(r,c)\in F_{\omega,\omega}$, $\mathcal{O}\subseteq (\mathcal{A}{F_{\omega,\omega}})_{n+1}$, $(t,i)\in F_{M,N}$, for some $M$, fixed $N=n$ and $[(t,i),(r,c)]\neq \emptyset$. There exists $(r',c')\in [(t,i),(r,c)]$ such that: 
 $$u_{n+1}[(t,i),(r',c')]\subseteq \mathcal{O} \text{ or } u_{n+1}[(t,i),(r',c')]\subseteq \mathcal{O}^{c},$$ 
 where  $u_{n+1}[(t,i),(r',c')]=\{\, (t',i')\in (r',c')_{n+1}: (t,i)\sqsubseteq (t',i')[n]\,\}$ and $(r,c)_{n+1}=\{\,(t,i): (t,i)=(r',c')[n+1], (r',c')\in [(r,c)] \, \}.$ We remind the reader here that given $(t,i)\in F_{M,N}$ and $(t',i')\in F_{M',N'}$, connections of finite length, with $M\leq M'$ and $N\leq N'$, by $(t,i)\sqsubseteq (t',i')$ we denote that $(t,i)=(t'\upharpoonright M,i'\upharpoonright N)$.
 
 \begin{proof}
 
 Consider the open set $[(t,i),(r,c)]\neq \emptyset$. By $A.3(1)$ and the definition of basic open sets of $F_{\omega,\omega}$, we can assume that $(t,i) \preceq (r,c)$.  
  Let $c:[(t,i),(r,c)]\to 2$ be a two coloring, defined by $$c(r',c')=\mathcal{X}_{\mathcal{O}}((r',c')[n+1]),$$ where $\mathcal{X}_{\mathcal{O}}$ is the characteristic function of the set $\mathcal{O}\subseteq  (\mathcal{A}F_{\omega,\omega})_{n+1}$. This is obviously a Borel coloring, with respect the $\tau_2$ topology.  Theorem $3$ applies, where $A$ an alphabet with $|A|=n$, to give us $(r_0,c_0)\in F_{\omega, \omega}$ so that $c$ is monochromatic on $[(t,i),(r_0,c_0)]$. For the basic open set $[(t,i),(r_0,c_0)]$ one has that either $$u_{n+1}[(t,i),(r_0,c_0)] \subseteq \mathcal{O}\text{ or }u_{n+1}[(t,i),(r_0,c_0)]\subseteq \mathcal{O}^{c}.$$
 Therefore $(r_0,c_0)$ satisfies the claim of axiom $A.4$.
 
  \end{proof}

Considering ${\mathcal{A}F}_{\omega, \omega}$ as a discrete space, the infinite power ${\mathcal{A}F}_{\omega, \omega}^\mathbb{N}$ gets its Tychonov product topology, a completely metrizable topology. We consider $F_{\omega, \omega}$ a subset of ${\mathcal{A}F}_{\omega, \omega}^\mathbb{N}$ via the identification $(r,c)\to (u_n(r,c))_{n\in \omega}$. Thus it is natural to call $F_{\omega, \omega}$ closed if in this identification it corresponds to a closed subset of ${\mathcal{A}F}_{\omega, \omega}^\mathbb{N}$. It is obvious that $F_{\omega, \omega}$ is closed.

Now we can state the Theorem $5.4$ from \cite{To} in the context of $F_{\omega, \omega}$.

\begin{mtheo}  The fact that $\langle F_{\omega,\omega}, u, \leq \rangle$ is closed and it satisfies axioms $A.1$, $A.2$, $A.3$ and $A.4$ implies that every property of Baire subset of $F_{\omega, \omega}$ is Ramsey and every meager subset is Ramsey null. Moreover the field of Ramsey subsets, that coincides with the field of Baire measurable subsets, is closed under the Suslin operation.

\end{mtheo}

Any space that satisfies the conclusions of the above theorem is called a topological Ramsey space. 
That Ramsey and topological equivalence in any topological Ramsey space gives as a corollary a Ramsey theorem. In the case of our topological Ramsey space $\langle F_{\omega,\omega}, u, \leq \rangle$ is the following:

\begin{theorem}
Given a finite Suslin measurable coloring $g:[(t,i),(r',c')] \to l$, there exists a $(r,c)\in [(t,i),(r',c')]$ such that $[(t,i),(r,c)]$ is monochromatic.
\end{theorem}

  $\langle F_{\omega,\omega}, u, \leq \rangle$ forms a topological Ramsey space, so its field of Baire measurable subsets coincides with that of Ramsey and is closed under the Suslin operation. Therefore for any finite coloring, where each color is Suslin measurable, the assertion of the above theorem follows immediately. Notice that Theorem $4$ is a reformulation of Theorem $2$ for $(t,i)=( \emptyset,\emptyset)$.
  
   Now we prove the corresponding version of Theorem $2$ for $F_{\omega, K}$, where $K<\omega$. We introduce a topology on $F_{\omega,K}$ with basic open sets: $[(t,i),(r,c)]=\{ (s,j)\in F_{\omega,K} : (t,i)\sqsubseteq (s,j), (s,j)\leq (r,c)\}$, where $(t,i)\in F_{M,L}$, $M\in \omega$, $L\leq K$ and $(r,c)\in F_{\omega,\omega}$.

   \begin{theorem}\label{maintheorem} Let $l>0$ be a natural number. For each finite Suslin measurable $l$ coloring of $F_{\omega, K}$,  $g:F_{\omega,K}\to l$, there exists $(r,c) \in F_ {\omega, \omega}$ such that the set $$ F_{\omega,K}\cdot (r,c)=\{ \, (s,j)\cdot(r,c): (s,j)\in F_{\omega, K}\, \}$$ is monochromatic.
     \end{theorem}
  
   \begin{proof}
   Consider the connection $(s_0,j_0)\in F_{\omega, K}$ defined as follows: $s_0\upharpoonright K=id_K$ and $s_0\upharpoonright [K,\omega)=0$. The way $s_0$ is defined implies that $j_0\upharpoonright K=id_K$.
   Let $\pi:F_{\omega,\omega} \to F_{\omega,K}$ be defined by $\pi((r,c))=(s_0,j_0)\cdot (r,c)$.     Consider now the composition $g\circ \pi: F_{\omega,\omega}\to l$ which is also Suslin measurable. By Theorem $2$ there exists $(r',c')\in F_{\omega,\omega}$ such that $[(r',c')]$ is monochromatic with respect the above composition.

      Let $(r,c)=(r_1,c_1) \cdot (r',c')$, where $(r_1,c_1)$ has the property that for all $n\in \omega$ the set $r_1^{-1}(\{n\})$ is of infinite cardinality.   
   Notice that any $(s,j)\in  F_{\omega,K}\cdot (r,c)$ can be written as $\pi (r'',c'')$ for some $(r'',c'')\in [(r,c)]$. Therefore $(r,c)$ is such that $ F_{\omega,K}\cdot (r,c)$ monochromatic with respect the coloring $g$.
  \end{proof}

 \begin{center}
  \section{Baire measurability}
  \end{center}

  In this section we show that Theorem $5$ does extend to the realm of Baire measurable colorings relative to an appropriate topology on $F_{\omega,K}$, $K<\omega$.    
    
       Consider the completely metrizable topology on the space $F_{\omega,K}$ having as basic open sets, sets of the form: 
                  \begin{center} 
    $[(t,i)]=\{ \, (s,j)\in F_{\omega,K}: (t,i)\sqsubseteq (s,j) \,\}$,
                  \end{center}
                  
  where $(t,i)\in F_{L,K}$ for $K\leq L<\omega$. We remind the reader that by $(t,i)\sqsubseteq (s,j)$ we denote that $(t,i)=(s\upharpoonright L,j\upharpoonright K)$.\\

    \begin{theorem}
    Let $g:F_{\omega,K}\to l$ be a finite coloring that is Baire measurable relative to the metrizable topology defined just above. Then there exists an $(r,c)\in F_{\omega,\omega}$ such that the set $[(r,c)]_K=\{\, (s,j)\in F_{\omega,K}: (s,j)\leq (r,c)\,\}$ is $g$-monochromatic.
    \end{theorem}

    \begin{proof}
    The coloring $g$ is Baire measurable and $F_{\omega,K}$ is completely metrizable and second countable, so there is a dense $G_{\delta}$ subset $\mathcal{G}$ of $F_{\omega,K}$ where the coloring is actually continuous. 
       We call a sequence $(D_p)_{p\in \omega}$ of finite subsets of $\mathbb{N}$ a block sequence, if for any two $p_0$, $p_1$ with $p_0<p_1$ we have $\max D_{p_0}< \min D_{p_1}$.\\
    To continue we need the following:
      \begin{lemma}
   Given a dense $G_\delta$ subset $\mathcal{G}$ of $F_{\omega,K}$ and $(t,i) \in F_{L,K}$, there is an infinite block sequence $(D_p)_{p\in \omega}$ of finite subsets of $\mathbb{N}$ and for each $p$ a surjective mapping $f_p:D_p\to K$ such that for every $(s,j)\in F_{\omega, K}$ if $(t,i)=(s\upharpoonright L,j\upharpoonright K)$ and $s$ extends infinitely many of the $f_p$, then  $(s,j)\in \mathcal{G}$.
    \end{lemma}

    \begin{proof}
    $\mathcal{G}$ is a dense $G_\delta$ subset of $F_{\omega,K}$, so its complement is a countable union of closed nowhere dense sets namely $\mathcal{G}^{c}=\cup_{n\in \omega}N_n$. Start with the open set $[(t,i)]$ and $N_0$, there is an extension $t_0$ of $t$ such that $[(t_0,i)]\cap N_0=\emptyset$. Set ${d_0}= dom(t_0\setminus t)$ and $g_0:d_0\to K$ defined by $g_0=t_0\upharpoonright d_0$. We remark here that $i$ is fixed so no matter how we extend $t$ has not impact on $i$.  We  could think of a rigid surjection defined on a finite interval, in a set theorytic way, as a set of ordered pairs, so given two rigid surjections defined on two disjoint intervals we can consider their union, namely the rigid surjection defined on both intervals and is equal to each of its components when restricted to one of them.

     Let $\{\,{t^0_h}:h\in m\, \}$ where $t^0_0=t_0$, be an enumeration of all possible extensions of $t$ with length equal to $|t_0|$. Consider the open set $[(t^0_1,i)]$ and let $t_1$ be an extension of the rigid surjection $t^0_1$ such that $[(t_1,i)]\cap N_0=\emptyset$. Set now ${d_1}=dom(t_1\setminus t^0_1)$ and $g_1:d_1\to K$ defined by $g_1=t_1\upharpoonright d_1$. Then consider the open set $[(t^0_2 \cup g_1,i)]$ and let $t'_2$ be an extension of the rigid surjection $t^0_2 \cup g_1$ such that $[(t'_2,i)]\cap N_0=\emptyset$. Once more set  $d_2=dom(t'_2\setminus t^0_2 \cup g_1)$ and $g_2:d_2\to K$ defined by $g_2=t'_2\upharpoonright d_2$. After $m$-steps we have defined $({d_i})_{i\in m}$ and similarly $(g_i)_{i\in m}$. Set now $D_0=d_1\cup \dots \cup d_{m-1}$ and $f_0=g_1\cup \dots \cup g_{m-1}$. They have the property that for all $t^0_h$ with $h< m$ $$ [(t^0_h\cup f_0,i)]\cap N_0=\emptyset.$$
    
    Consider now the open set $[(t\cup g_0\cup  f_0,i)]$ and let $t_2$ be an extension of the rigid surjection $ t\cup g_0\cup f_0$ such that $[(t_2,i)]\cap N_1=\emptyset$. Observe that $[(t_2,i)]\cap N_0=\emptyset$ as well.
   
     Let $\{\, {t^1_h}:h\in  n\,\}$, where $t^1_0=t_2$, be an enumeration of all possible extensions of $t$ with length equal to $|t_2|$. As above start off with $t^1_1$ and consider the open set $[( t^1_1 ,i)]$. Let $g_{m}: d_{m}\to K$ be a such that $[(t^1_1 \cup g_{m}),i)]\cap (N_0\cup N_1)=\emptyset$. After $n$ steps we have created $(d_h)_{h\in [m,m+n-1]}$ and their corresponding mappings $(g_h)$. Set $D_1=d_{m}\cup \dots \cup d_{(m+n-1)}$ and $f_1=g_{m}\cup \dots \cup g_{(m+n-1)}$. Then $D_1$ and $f_1$ have the property that for all $h\in n\,[(t^1_h\cup f_1,i)]\cap (N_0\cup N_1)=\emptyset$.

     We proceed in that manner to get $(D_p)_{p\in \omega}$ and $f_p:D_p\to K$ their corresponding surjections. Now suppose that $(s,j)\in F_{\omega,K}$ is such that $(t,i)=(s\upharpoonright L,j\upharpoonright K)$, $s$ extends infinitely many of $f_p$ and $(s,j)\in \mathcal{G}^{c}$. Then $(s,j)\in N_n$ for some $n$. But $s$ extends in particular some $f_m$ where $m \geq n$ which implies that $(s,j)$ does not belong to the finite union $N_0\cup \dots \cup N_m$ a contradiction.
      \end{proof}
      
 We continue now the proof of Theorem $6$. Let $(b_l)_{l\in \mathbb{N}}$ be an enumeration of all $(t,i)\in F_{L,K}$, $L \in \omega$, i.e. all initial segments of all elements of $F_{\omega,K}$ where the domains of their increasing injections are all equal to $K$. For each $l$, let $(f^l_p)_{p\in \mathbb{N}}$ be the sequence of block mappings  given by Lemma $1$ when applied to $\mathcal{G}$ and $b_l$. We will create $(r',c')\in F_{\omega,\omega}$ with the property that $[(r',c')]_{K}\subseteq \mathcal{G}$.\\
 
 \begin{claim} There exists $(r',c')\in F_{\omega,\omega}$ such that $[(r',c')]_K\subseteq \mathcal{G}$. 
 
 \end{claim}
 
 From this claim, we see that $g \upharpoonright [(r',c')]_K$ is continuous so we apply Theorem $5$ to get the desired $(r,c)$ that satisfies the conclusions of Theorem $6$.\\
 
 It rests to prove Claim $1$.
 \begin{proof}
  We build $(r',c')\in F_{\omega,\omega}$ recursively by deciding its restrictions on $(r',c')[n]$ for $n\in \omega$. We aim to create $(r',c')$ so that for every $n\in \omega$, $| r'^{-1}(\{n\})|=\omega$.
  Start with a $(t_0,i_0)\in F_{M,K}$ for some $M\in \omega$. We extend $t_0$ by $t_1$ with $dom(t_1)=  M+K+1$. The extension is defined as follows: $t_1\upharpoonright M=t_0 \upharpoonright M$, $t_1(M+k)=k$ for $k<K$ and $t_1(M+K)=K$. Similarly we extend $i_0$ by $i_1$ as follows: $i_1\upharpoonright K=i_0 \upharpoonright K$, $i_1(K)=M+K$, and we set $E_{K+1}=M+K+1$. Notice that $(t_1,i_1)\in F_{E_{K+1}, K+1}$.

   Next consider the finite set: $$A_{K+1}= \{ (t_0,i_0)\}\cup \{ \, (b_q)_{q\in m-1} :b_q\in F_{E_{K+1},K}, b_q\leq_{fin} (t_1,i_1) \}.$$ Recall that, as we remarked in the first section, rigid surjections are nothing more than equivalence relations on their domain. For notational convenience, let $b_0=(t_0,i_0)$, so $(b_q)_{q\in m}$.
   Start off with $b_0$, choose a finite partial mapping $f_0:D_0\to K$ associated with $b_0$ such that $D_0$ lies above $E_{K+1}$ i.e. $\min D_0 > E_{K+1}$. Set $C_0=\max(D_0)+1$ and define the extension $(t_1^0,i_1^0)$ of $(t_1 ,i_1)$ as follows: $f_0^{-1}(m)$ gets in the same equivalent class, by $t_1^0$, as the minimal point of $ b_0^{-1}(m)$, for $m\in K$.
  In the case that $(C_0\setminus D_0)\neq \emptyset$ we require that $t_1^0 \upharpoonright (C_0\setminus D_0)\subseteq K+1$ respecting only the conditions of $t_1^0$ being a rigid surjection. Observe that $i_1^0=i_1\upharpoonright K+1$ is already defined. As a result $(t_1^0,i_1^0)\in F_{C_0, K+1}$.
  Consider $b_1$ now, and repeat the above step to get $f_1:D_1\to K$, with $D_1$ lying above of $C_0$. Set $C_1=\max(D_1)+1$. In this case $(t_2^0,i_2^0)$ extends $(t_1 ,i_1)$ as follows: $t_2^0\upharpoonright  C_0 =t_1^0$ and $t_2^0$ gets $f_1^{-1}(m)$ in the same class as the minimal element of $b_1^{-1}(m)$. Once again we define $t_2^0$ on $C_1\setminus D_1$ arbitrarily respecting only the conditions of $t_2^0$ being a connection.
  After $m$ steps we finally define $C_{m-1}$ and the extension of $(t_1,i_1)$ to $(t_{m-1}^0,i_{m-1}^0)$. Observe that $i_1=i_{m-1}^0$ and $dom(t_{m-1}^0)=C_{m-1}$. We set $E_{K+2}= C_{m-1}+K+2$ and define $(t_2,i_2)$ an extension of $(t_{m-1}^0,i_{m-1}^0)$  as follows: $t_2\upharpoonright C_{m-1}=t_{m-1}^0\upharpoonright C_{m-1}$, $t_2(C_{m-1}+ k)=k$, for $k\in K+1$, and $t_2(C_{m-1}+K+1)=K+1$. Let $i_2(K+1)=C_{m-1}+K+1$ and $i_2\upharpoonright K+1=i_{m-1}^0\upharpoonright K+1$. Notice that $(t_1,i_1)=(t_2\upharpoonright C_{m-1}+K+1,i_2\upharpoonright K+1)$. We denote that by $(t_1,i_1)\sqsubseteq(t_2,i_2)$. Observe also that $(t_2,i_2)\in F_{E_{K+2}, K+2}$.
  Next  consider the finite set:  $$A_{K+2}=A_{K+1}\cup \{  (b_q)_{q\in w} : b_q\in F_{E_{K+2},K},  b_q \leq_{fin} (t_2,i_2)  \}.$$
  
  Proceed as above. Let $(r',c')[n]=(t_1,i_1)[n]$, for $n= K-1$ and $(r',c')[n]= (t_{n-K+2},i_{n-K+2})$, for $n\geq K$. We have to demonstrate that $(r',c')$ has the desired property.

  Consider an arbitrary $(s',j')\in [(r',c')]_K=\{\, (s,j)\in F_{\omega,K}: (s,j)\leq (r',c')\,\}$. Let $ (s'\upharpoonright M',j'\upharpoonright K)=(t',i')\in F_{M',K}$. The rigid surjection $s'$ extends $t': M' \to K$ and infinitely many of the $f_p^l$ with $b_l=(t',i')$. To see this notice that $b_l$ is contained in every set $A_{K+N}$ with $M' \leq E_{K+N}$. As a result $(s',j')\in \mathcal{G}$. 
    \end{proof}
  \end{proof}
  
  Therefore our Theorem $5$ holds in the realm of Baire measurable colorings as well.

\section{ Acknowledgement}

I must thank the referee for his careful reading, his patience and his valuable suggestions.

     \end{document}